\documentclass[reqno]{amsart}

\usepackage{hyperref}
\usepackage{mathtools}
\usepackage{mathrsfs}
\usepackage{amsmath}

\newtheorem{lemma}{Lemma}[section]
\newtheorem{theorem}[lemma]{Theorem}

\theoremstyle{definition}
\newtheorem{remark}[lemma]{Remark}

\newtheorem{conjecture}[lemma]{Conjecture}
\newtheorem{question}{Question}

\usepackage{amssymb}
\renewcommand{\leq}{\leqslant}
\renewcommand{\geq}{\geqslant}
\newcommand{\eps}{\epsilon}

\def\SO{\text{SO}}

\def\F{\mathbf F}

\def\C{\text{C}}
\newcommand{\SL}{\mathrm{SL}}
\newcommand{\SU}{\mathrm{SU}}
\newcommand{\PGL}{\mathrm{PGL}}
\newcommand{\Aut}{\mathrm{Aut}}
\newcommand{\Out}{\mathrm{Out}}

\newcommand{\SOr}{\mathrm{SO}}

\newcommand{\Sp}{\mathrm{Sp}}
\newcommand{\POmega}{\mathrm{P\Omega}}
\newcommand{\PSL}{\mathrm{PSL}}
\newcommand{\PSU}{\mathrm{PSU}}
\newcommand{\PSp}{\mathrm{PSp}}

\newcommand{\PGammaL}{\mathrm{P\Gamma L}}
\newcommand{\PGammaU}{\mathrm{P\Gamma U}}

\newcommand{\Soc}{\mathrm{Soc}}

\usepackage{todonotes}

\newcommand{\daniele}[1]{\todo[color=green!40]{Daniele: #1}{}}

\begin{document}

\title{On the number of conjugacy classes of a primitive permutation group with nonabelian socle}

\author{Daniele Garzoni}
\address{Daniele Garzoni, Dipartimento di Matematica ``Tullio Levi-Civita'', Universit\`a degli Studi di Padova, Padova, Italy}
\email{daniele.garzoni@phd.unipd.it}

\author{Nick Gill}
\address{Nick Gill, Department of Mathematics, University of South Wales, Treforest CF37 1DL, U.K.}
\email{nick.gill@southwales.ac.uk}

\maketitle

\begin{abstract}
    Let $G$ be a primitive permutation group of degree $n$ with nonabelian socle, and let $k(G)$ be the number of conjugacy classes of $G$. We prove that either $k(G)<n/2$ and $k(G)=o(n)$ as $n\rightarrow \infty$, or $G$ belongs to explicit families of examples.
\end{abstract}

\section{Introduction}

Throughout, $k(G)$ denotes the number of conjugacy classes of a finite group $G$. Mar\'oti \cite{marotibounding} proved that if $G$ is a primitive permutation group $G$ of degree $n$, then $k(G)\leq p(n)$, where $p(n)$ denotes the number of partitions of $n$. This bound is attained by $S_n$ in its action on $n$ points. Moreover, he proved that if the socle of $G$ is not a direct product of alternating groups, then $k(G)\leq n^6$.

In this paper, we want to improve this bound under the assumption that $G$ has nonabelian socle. In Subsection \ref{subsec: context}, we will give more context and review more results in this area.

There are two special types of primitive groups which we wish to single out.

\begin{enumerate}
    \item[(A)] Let $G$ be the symmetric group $S_d$ or the alternating group $A_d$ on $d\geq 5$ letters. For every $1\leq k< d/2$, $G$ acts primitively on the set of $k$-subsets of $\{1, \ldots, d\}$. These are in number $\binom{d}{k}$.
    \item[(B)] Let $G$ be an almost simple group with socle $\PSL_d(q)$, and assume that $G\leq \PGammaL_d(q)$. Then $G$ acts primitively on the set of $1$-subspaces of $\F_q^d$. These are in number $(q^d-1)/(q-1)$.
\end{enumerate}

Our main result says that, if $G$ is a primitive group with nonabelian socle, then either $G$ has very few conjugacy classes, or else the action of $G$ is ``related'' to (A) or (B) or to a further finitely many almost simple primitive permutation groups. The precise statement is as follows.

\begin{theorem}
\label{t: conjugacyclassesprimitive}
Let $G$ be a primitive permutation group of degree $n$ with nonabelian socle, so $\Soc(G)\cong S^r$, with $S$ nonabelian simple and $r\geq 1$. Then one of the following holds.
\begin{enumerate}
    \item $k(G)<n/2$, and $k(G)=O(n^\delta)$ for some absolute $\delta <1$.
    \item $G\leq A\wr S_r$, $A$ is an almost simple primitive permutation group of degree $m$ with socle $S$, $G$ acts in product action on $n=m^r$ points, and one of the following holds:
    \begin{itemize}
        \item[(i)] The action of $A$ on $m$ points is equivalent to an action in Table \ref{tab: final_exceptions}, and $k(G)<n^{1.31}$.
        \item[(ii)] The action of $A$ on $m$ points is isomorphic to an action described in (A) or (B). In the (B)-case, $k(G)<n^{1.9}$.
    \end{itemize}
\end{enumerate}
\end{theorem}

\begin{table}
    \centering
    \begin{tabular}{ccc}
    \hline\noalign{\smallskip}
         $G$ & $n$ & $k(G)$  \\
          \noalign{\smallskip}\hline\noalign{\smallskip}
          $M_{11}$ & 11,12 & 10 \\
          $M_{12}$ & 12,12 & 15 \\
          $M_{22}$ & 22 & 12 \\
          $M_{22}.2$ & 22 & 21 \\
          $M_{23}$ & 23 & 17 \\
          $M_{24}$ & 24 & 26 \\
          $A_7$ & 15,15 & 9 \\
          $S_8\cong \SL_4(2).2$ & 35 & 22 \\
          $\PSL_2(11)$ & 11,11 & 8 \\
          $\SOr_8^-(2)$ & 119 & 60 \\
           $\Sp_8(2)$ & 120,136 & 81 \\
           $\SOr_8^+(2)$ & 120 & 67 \\
           $\Sp_6(2)$ & 28, 36 & 30 \\
           $\PSp_4(3)\cong\SU_4(2)$ & 27,36,40,40 & 20 \\
           $\PSp_4(3).2$ & 27,36,40,40,45 & 25 \\
           $\PSU_4(3).(2\times 2)$ & 112 & 59 \\
           $\PGammaU_4(3)$ & 112 & 61 \\
           $\SU_3(3)$ & 28 & 14 \\
           $\SU_3(3).2$ & 28 & 16 \\
         \noalign{\smallskip}\hline\noalign{\medskip}
         
    \end{tabular}
    \caption{Almost simple primitive permutation groups $G$ of degree $n$ (up to equivalence) for which $k(G)\geq \frac n2$, and for which the action is not isomorphic to an action in (A) or (B).}
    \label{tab: final_exceptions}
\end{table}

We will first prove Theorem \ref{t: conjugacyclassesprimitive} in case $G$ is almost simple, and then deduce the general case. For convenience, we state separately the almost simple case (where we also give an explicit estimate for $\delta$).

\begin{theorem}
\label{t: main_almost_simple}
Let $G$ be an almost simple primitive permutation group of degree $n$. Then one the following holds.
\begin{enumerate}
    \item $k(G)<n/2$, and $k(G)=O(n^{3/4})$.
    \item Either the action of $G$ is equivalent to an action in Table \ref{tab: final_exceptions}, or the action of $G$ is isomorphic to an action described in (A) or (B). In the (B)-case, $k(G)< 100n$.
\end{enumerate}
\end{theorem}

In item (1), the exponent $3/4$ is sharp, although in most cases $k(G)=o(n^{3/4})$ as $n\rightarrow \infty$; see Remark \ref{r: littleo(n34)} for a precise statement.

In the proof of Theorem \ref{t: main_almost_simple}, an essential ingredient is the work of Fulman--Guralnick \cite{FG1}, which gives upper bounds for the number of conjugacy classes of almost simple groups of Lie type. 


We immediately make some clarifications regarding the statement of Theorem \ref{t: conjugacyclassesprimitive}.

\begin{remark}
\label{r: clarifications}

    
(i) We are not asserting that every case appearing in Theorem \ref{t: conjugacyclassesprimitive}(2) does not satisfy item (1). For instance, assume that $m=n$, and consider $G=S_d$ acting on $n=\binom{d}{k}$ points as in (A), and assume that $cd\leq k\leq \frac d2$ for some fixed constant $c$. Then it is well known that $n=\binom{d}{k}$ is exponential in $d$, while the number of conjugacy classes of $G=S_d$ is of the form $O(1)^{\sqrt d}$. In particular $k(G)=n^{o(1)}$ as $d\rightarrow \infty$.
    
(ii) In Theorem \ref{t: conjugacyclassesprimitive}(2)(ii), we can be more precise about the adjective \textit{isomorphic}, as follows. If $A$ is $A_d$ or $S_d$, then either the action of $A$ is \textit{equivalent} to the action on $k$-subsets; or else $(d,m)=(6,6)$ or $(6,15)$. Moreover, if $A$ is almost simple with socle $\PSL_d(q)$ and $A\leq \PGammaL_d(q)$, then the action of $A$ is equivalent to the action on the $1$-subspaces or $(d-1)$-subspaces of $\F_q^d$. For this, see Lemmas \ref{l: alternating} and \ref{l: psl}.
    
(iii) Whenever $G$ is almost simple with socle isomorphic to both $A_d$ and $\PSL_f(q)$, we have excluded from Table \ref{tab: final_exceptions} both the groups in (A) and (B). For instance, $G=S_6$ has $11$ conjugacy classes, and contains a subgroup $S_3\wr S_2$ of index $10$ acting transitively on $6$ points; but this does not appear in Table \ref{tab: final_exceptions} in view of the isomorphism $S_6\cong \text P\Sigma\text L_2(9)$. The same reasoning applies to the isomorphisms $\SL_2(4)\cong \PSL_2(5)\cong A_5$, $\PSL_2(7)\cong \SL_3(2)$ and $\SL_4(2)\cong A_8$.

\end{remark}

\subsection{When is \texorpdfstring{$k(G)=o(n)$}{k(G)=o(n)}?} Theorem \ref{t: conjugacyclassesprimitive} implies in particular that, if the socle of $G$ is nonabelian, then either $k(G)=o(n)$, or $G$ is ``known''. Can we prove that $k(G)=o(n)$ in further cases?

We are particularly interested in the cases contemplated in Theorem \ref{t: conjugacyclassesprimitive}(2)(i), for which we show $k(G)< n^{1.31}$. We first note that there are examples in which $k(G)> n^{1.08}$ for arbitrarily large $n$, in contrast to item (1); see Lemma \ref{r: example_M12}.

Still, it would be interesting to understand 
precisely when this happens (since there are only finitely many almost simple groups to handle).

\begin{question}
\label{q: o(n)}
Let $A$ be an almost simple primitive group on $m$ points appearing in Table \ref{tab: final_exceptions}. Determine whether \textit{every} primitive subgroup $G$ of $A\wr S_r$ on $n=m^r$ points is such that $k(G)=o(m^r)$ as $r\rightarrow \infty$. 
\end{question}

This should be related to estimating the number of conjugacy classes in wreath products, and we refer to Section \ref{sec: final_comments} for comments in this direction. See in particular Conjecture \ref{conj_o(n)}, which would provide an answer to Question \ref{q: o(n)}.

\subsection{Context}
\label{subsec: context}
There are many results in the literature which give upper bounds to the number of conjugacy classes of a finite groups in terms of various parameters. We recall some of these, focusing on permutation groups.

Kov\'acs--Robinson \cite{kovacsrobinson} proved that every permutation group of degree $n$ has at most $5^{n-1}$ conjugacy classes. This estimate was subsequently improved by Liebeck--Pyber, Mar\'oti, and Garonzi--Mar\'oti, as follows:
\begin{align*}
    &k(G)\leq 2^{n-1} &&\text{(\cite{liebeckpyber})} \\
    &k(G)\leq 3^{(n-1)/2} \,\text{ for } n\geq 3 &&\text{(\cite{marotibounding})} \\
    &k(G)\leq 5^{(n-1)/3} \,\text{ for } n\geq 4 &&\text{(\cite{garonzimaroti})}.
\end{align*}
We should mention that, in \cite{kovacsrobinson} and \cite{liebeckpyber}, various other upper bounds to $k(G)$ are proved, where $G$ is not necessarily a permutation group.

There are easy examples showing that these estimates are somewhat close to best possible, even for transitive groups. Indeed, the subgroup $S_4^{n/4}\leq S_n$ has $5^{n/4}$ conjugacy classes; and the transitive subgroup $G=S_4\wr C_{n/4}\leq S_n$ has $5^{n/4-o(n)}$ conjugacy classes (more precisely, $k(G)$ is asymptotic to $4\cdot 5^{n/4}/n$; see Lemma \ref{l: regular}). 

For primitive groups, the situation is very different. Improving results from \cite{liebeckpyber}, Mar\'oti \cite{marotibounding} proved that every normal subgroup of a primitive permutation group $G$ of degree $n$ has at most $p(n)$ conjugacy classes; and if the socle of $G$ is not a direct product of alternating groups, then $k(G)\leq n^6$. (Recall that $p(n)=O(1)^{\sqrt n}$, and in fact the asymptotic behaviour of $p(n)$ is known by famous work of Hardy--Ramanujan.) Theorem \ref{t: conjugacyclassesprimitive} can be regarded as an improvement of this statement, for the case where the socle of $G$ is nonabelian.

\subsection{Abelian socle}
\label{subsec: abelian_socle}
In this paper we do not address the case in which the socle of $G$ is abelian. In this case, we still have the bound $k(G)\leq n^6$ from \cite{marotibounding}.

There is a deep problem, known as the \textit{non-coprime $k(GV)$-problem}, which was addressed by Guralnick--Tiep \cite{guralnicktiep} and which asks (in particular) for a characterization of the affine primitive permutation groups of degree $n$ for which $k(G)>n$. A resolution of this problem, if combined with the main result of this paper, would give a characterization of all primitive permutation groups of degree $n$ for which $k(G)>n$. We refer to Guralnick--Tiep \cite{guralnicktiep}, Guralnick--Mar\'oti \cite{guralnickmaroti} and the references therein for results in this direction, partly motivated by the celebrated Brauer's $k(B)$-conjecture.

The organization of the paper is as follows. In Section \ref{sec: proof_almost_simple} we prove Theorem \ref{t: main_almost_simple}, in Section \ref{sec:general case} we prove Theorem \ref{t: conjugacyclassesprimitive}, and in Section \ref{sec: final_comments} we discuss Question \ref{q: o(n)} and make further comments.


\section{Almost simple groups}
\label{sec: proof_almost_simple}
In this section we prove Theorem \ref{t: main_almost_simple}. Regarding item (1), we prove the inequality $k(G)<n/2$ in Subsections \ref{subsec: sporadic}--\ref{subsec: lie_type}, and then we prove the asymptotic inequality $k(G)=O(n^{3/4})$ in Subsection \ref{subsec: proof_thm_almost_simple}.

First, we gather some results that we will use throughout. 

\subsection{Some preliminary results and notation}
For a finite group $G$, let $P(G)$ be the minimal degree of a faithful permutation representation of $G$. If $G$ is almost simple with socle $S$, then $P(G)$ coincides with the minimal degree of a faithful transitive permutation representation of $G$, and moreover $P(S)\leq P(G)$. The values of $P(G)$ for $G$ a finite simple group are known; they are listed for instance in \cite[Table~4]{GMPS}. 

We begin with a lemma from \cite{gallagher}. 
We will often apply this lemma with no mention.

\begin{lemma}\label{l: sub conj}
If $G$ is a finite group and $H$ is a subgroup of $G$, then
\[
k(H)/|G:H|\leq k(G) \leq |G:H|\cdot k(H).
\]
If moreover $H$ is normal in $G$, then
\[
k(G)\leq k(H)\cdot k(G/H).
\]
\end{lemma}

In one occasion, we will need the following variant (see \cite[p. 447]{kovacsrobinson}).

\begin{lemma}
\label{l: improvementclasses}
Let $G$ be a finite group and let $N$ be a normal subgroup of $G$. Then
\[
k(G)\leq |G:N|\cdot \# \{G\text{-conjugacy classes of} \,\, N\}.
\]
\end{lemma}

We are now ready to begin the proof of Theorem \ref{t: main_almost_simple}.

\subsection{Sporadic groups}
\label{subsec: sporadic}

\begin{lemma}\label{l: sporadic}
Let $G$ be almost simple with socle $S$, a sporadic simple group. Let $M$ be a core-free maximal subgroup of $G$, and write $n=|G:M|$. If $k(G)\geq \frac{n}{2}$, then $G$ and $n$ are listed in Table~\ref{t: sporadic}.

\end{lemma}

Note that repeated values of $n$ in Table~\ref{t: sporadic} signify the existence of more than one action, up to equivalence, of the given degree. The same convention applies to later tables pertaining to the alternating groups and the groups of Lie type.

\begin{table}
\centering
\begin{tabular}{ccc}
\hline\noalign{\smallskip}
$G$ & $n$ & $k(G)$  \\
\noalign{\smallskip}\hline\noalign{\smallskip}
$M_{11}$ & 11,12 & 10 \\ 
$M_{12}$ & 12,12 & 15 \\
$M_{22}$ & 22 & 12 \\
$M_{22}.2$ & 22 & 21\\
$M_{23}$ & 23 & 17 \\
$M_{24}$ & 24 & 26 \\
\noalign{\smallskip}\hline\noalign{\medskip}

\end{tabular}
\caption{Faithful primitive permutation representations of degree $n$ for sporadic almost simple groups $G$ such that $k(G)\geq \frac{n}2$.}\label{t: sporadic}
\label{tab: sporadic}
\end{table}

\begin{proof}
We go through the ATLAS \cite{ATLAS}.
\end{proof}

\subsection{Alternating groups}
\label{subsec: alternating}
We recall some results that we will use. The first is an inequality of 
Pribitkin \cite{Pribitkin}, as follows.

\begin{lemma}\label{l: partition}
Let $p(d)$ be the number of partitions of the integer $d$. Then
\[
p(d)<\frac{e^{\pi\sqrt{2d/3}}}{d^{3/4}}.
\]
\end{lemma}

We will also need the following pair of inequalities which are an easy consequence of work of Robbins on the Stirling approximations \cite{robbins}.

\begin{lemma}\label{l: factorial}
Let $d\geq 2$ be an integer. Then
\[
\sqrt{2\pi} d^{d+1/2} e^{-d} \leq d! \leq e\, d^{d+1/2} e^{-d}.
\]
\end{lemma}

Finally we will need the following result of Praeger and Saxl which makes use of CFSG \cite{PraSax}.

\begin{lemma}\label{l: primitive}
Let $G\leq S_d$ and suppose that $G$ is primitive and does not contain $A_d$. Then $|G|<4^d$.
\end{lemma}

We now prove a lemma, which is known (see \cite[Corollary 2.7]{FG1}).

\begin{lemma}\label{l: alt sym}
For all $d$, $k(A_d)\leq k(S_d)$.
\end{lemma}

The proof can be easily modified to obtain a strict inequality for $d\geq 4$.

\begin{proof}
For $d\leq 9$ we can check this directly. Thus assume that $d>9$ and assume that $k(A_d)$ is equal to the number of even partitions of $d$ plus the number of partitions of $d$ into distinct odd numbers; on the other hand $k(S_d)$ is equal to the number of even partitions of $d$ plus the number of odd partitions of $d$. Thus we need to show that the number of partitions of $d$ into distinct odd numbers is less than or equal to the number of odd partitions of $d$.

Given $\ell\geq 10$, a positive even integer, observe that there are precisely $\lfloor\frac{\ell}{4}\rfloor$ ways to partition $\ell$ into two distinct odd numbers. On the other hand, let $Y_\ell$ be the set of partitions of $\ell$ into an odd number of even numbers. It is easy to see that $|Y_\ell|>\lfloor\frac{\ell}{4}\rfloor$: we simply use the partitions
\[
\ell, \, (\ell-4,2^2), \, (\ell-8,2^4),\dots \textrm{ and } (\ell-8, 4^2).
\]
Now let $X$ be the set of all partitions of $d$ into at least two distinct odd numbers; let $X_\ell$ be the subset of all partitions of $X$ for which the sum of the two largest parts is equal to $\ell$. Since $d\geq 10$, observe that $\ell\geq 10$. To each partition $P\in X_\ell$ we can associate a partition which has the same parts as $P$, apart from the largest two, which are replaced by a partition from $Y_\ell$. The counting above implies that we can do this in such a way that the association yields an injective function from $X$ to a proper subset of the odd partitions of $X$. 

If $d$ is even, then this yields that $k(A_d)<k(S_d)$. If $d$ is odd, we must also associate an odd partition with the partition of $d$ consisting of a single part. But since our injective function is not onto, this can be done. The result follows.
\end{proof}

Now we can prove the main result of this subsection.

\begin{lemma}\label{l: alternating}
Let $G$ be almost simple with socle $S\cong A_d$. Let $M$ be a core-free maximal subgroup of $G$, and write $n=|G:M|$. If $k(G)\geq \frac{n}{2}$, then one of the following holds.
\begin{enumerate}
    \item $G$ and $n$ are listed in Table~\ref{tab: alt ex}.
    \item $M$ is intransitive in its action on $d$ points, thus $n=\binom{d}{k}$ for some integer $k$ such that $1\leq k< \frac12d$.
    \item $(d,n)=(6,6)$ or $(6,15)$, and the action of $G$ on the cosets of $M$ is isomorphic, but not equivalent, to the action on the coset of a maximal intransitive subgroup.
\end{enumerate}
\end{lemma}
In item (3), if $(d,m)=(6,6)$ we have $G=A_6$ or $S_6$, and $M=S_5\cap G$, where $S_5$ is a subgroup of $S_6$ acting primitively on $6$ points. If $(d,m)=(6,15)$, again $G=A_6$ or $S_6$, and $M=(S_2\wr S_3)\cap G$, where $S_2\wr S_3$ acts transitively (and imprimitively) on $6$ points. (Note that in the latter case, if $G=A_6$ then $k(G)<n/2$.)

\begin{table}
    \centering
    \begin{tabular}{ccc}
    \hline\noalign{\smallskip}
          $G$ & $n$ & $k(G)$  \\
          \noalign{\smallskip}\hline\noalign{\smallskip}
$A_5$ & 6 & 5 \\
$S_5$ & 6 & 7 \\
$A_6=\PSL_2(9)$ & 10 & 7 \\
 $A_6.2=\PGL_2(9)$ & 10 & 11 \\
 $A_6.2=S_6$ & 10 & 11 \\
 $A_6.2=M_{10}$ & 10 & 8 \\
 $A_6.(2\times 2) = \PGammaL_2(9)$ & 10 & 13 \\
$A_7$ & 15,15 & 9 \\
$A_8$ & 15,15 & 14 \\
$S_8$ & 35 & 22 \\
          \noalign{\smallskip}\hline\noalign{\medskip}
    \end{tabular}
    \caption{Faithful primitive permutation representations of degree $n$ for almost simple groups $G$ with socle $A_d$ such that $k(G)\geq \frac n2$, and the action is not isomorphic to an action in (A).}
    \label{tab: alt ex}
\end{table}


\begin{proof}
For $d\leq 8$, we use the ATLAS \cite{ATLAS} together with GAP \cite{GAP} to obtain the given list. For $9\leq d \leq 20$ we use GAP to check that no examples occur. Assume, then, that $d>20$.

Let us suppose, first, that $M$ is primitive in its action on $d$ points. Then Lemmas~\ref{l: partition}--\ref{l: alt sym} imply that it is sufficient to prove the following
\[
\frac{e^{\pi\sqrt{2d/3}}}{d^{3/4}} < \frac{\sqrt{2\pi} d^{d+1/2}}{4 \cdot e^d \cdot 4^d}.
\]
If we assume that the other inequality holds, we get
\begin{align*}
    e^{\pi\sqrt{2d/3}} &\geq \frac{\sqrt{2\pi} \cdot d^{d+5/4}}{4 \cdot e^d \cdot 4^d} \\
\implies 2\cdot e^{\pi\sqrt{2d/3}} \cdot (4e)^d &\geq d^{d+5/4} \\
\implies  2\cdot e^{2.6\sqrt{d}} \cdot (4e)^d &\geq d^{d+5/4} \\
\implies  2\cdot (5e)^{d+\sqrt{d}} &\geq d^{d+5/4} \\
\implies  d &\leq 20.
\end{align*}
Since $d>20$, the result follows.

Let us suppose, next, that $M$ is imprimitive in its action on $d$ points. Then
\[
n=\frac{d!}{(k!)^\ell \ell!},\]
where $d=k\ell$ and $k, \ell \geq 2$. Now Lemma~\ref{l: factorial} implies that
\begin{align*}
    n &\geq \frac{\sqrt{2\pi}\cdot d^{d+1/2} \cdot e^{-d}}{(ek^{k+1/2}e^{-k})^\ell \cdot e\cdot \ell^{\ell+1/2}\cdot e^{-\ell}} = \frac{\sqrt{2\pi}\cdot d^{d+1/2}}{e\cdot k^{(k+1/2)\ell}\cdot \ell^{\ell+1/2}} \\
    & = \frac{\sqrt{2\pi}}{e}\cdot \frac{\ell^{d-\ell}}{k^{(\ell-1)/2}}\\
    & \geq \frac{\ell^{d-\ell}}{k^{\ell/2}}  \\ 
    & = \frac{\ell^{d-\ell}}{(d/\ell)^{\ell/2}},
\end{align*}
which implies that
\begin{align*}
    (d-\ell)\log(\ell)-\frac{\ell}{2}\log(d) + \frac{\ell}{2}\log(\ell) &\leq \log(n) \\
    \implies (d-\frac{\ell}{2})\log(\ell)-\frac{\ell}{2}\log(d) &\leq \log(n).
\end{align*}
If we fix $d$ and set $f(\ell)$ to be the function on the left hand side of the final inequality, with $\ell\in(0,d)$, then one computes that 
$f'(\ell)$ is a decreasing function. In particular, $f(\ell)$ takes its minimum value in the range $\ell\in(0,d)$ either when $\ell$ is as large as possible or as small as possible.

If $\ell$ is as small as possible, then $\ell=2$ and we obtain that
\[
f(2)\leq\log(n) \iff d-\log(d)-1\leq \log(n).
\]
If $\ell$ is as large as possible, then $\ell=\frac{d}{2}$ and we obtain that
\[
f\left(\frac{d}{2}\right)\leq\log(n) \iff \frac{d}{2}\log(d)-\frac{3d}{4}\leq \log(n).
\]
Now it is easy to check that, for $d\geq 4$, 
\[d-\log(d)-1 \leq \frac{d}{2}\log(d)-\frac{3d}{4},\]
and we conclude that
\begin{equation}
\label{eq:symm}
    d-\log(d)-1\leq\log(n).
\end{equation}   
On the other hand, if $k(G)\geq \frac{n}{2}$, then Lemmas~\ref{l: partition} and \ref{l: alt sym} imply that
\[
\frac n2 <\frac{e^{\pi\sqrt{2d/3}}}{d^{3/4}}.
\]
Taking logs and using \eqref{eq:symm}, we get 
\[
d < \frac 14 \log d + \pi \sqrt{\frac{2d}{3}}\log e + 2,
\]
which, since $d>20$, is false. This concludes the proof.
\end{proof}

\subsection{Groups of Lie type}
\label{subsec: lie_type}
For groups of Lie type, we will use results of Fulman and Guralnick giving bounds on the number of conjugacy classes \cite{FG1}.

\begin{theorem}\label{t: fg}
Let $G$ be a connected simple algebraic group of rank $r$ over a field of positive characteristic. Let $F$ be a Steinberg--Lang endomorphism of $G$ with $G^F$ a finite 
group of Lie type over the field $\F_q$. Then
\[k(G^F)\leq \min\{27.2q^r, q^r+68q^{r-1}\}.\]
\end{theorem}

\begin{theorem}\label{c: fg}
Let $G$ be an almost simple group with socle $S$, a 
simple group of Lie type of untwisted rank $r$ defined over $\F_q$. Then $k(G)\leq 100q^r$.
\end{theorem}

First we deal with exceptional groups.

\begin{lemma}\label{l: exceptional}
Let $G$ be almost simple with socle $S$, a simple exceptional group of Lie type. Then $k(G)< \frac12|G:M|$ for all core-free maximal subgroups $M$ of $G$.
\end{lemma}

\begin{proof}
We use the values for $P(S)$ given in \cite{GMPS}, as well as the fact that $k(G)\leq 100q^r$ (from Theorem~\ref{c: fg}). 

If $S$ is not a Suzuki group, then the value of $P(S)$ given in \cite{GMPS}  rules out all groups except $G_2(3)$, $G_2(4)$, $G_2(5)$, ${^3\!D_4}(2)$, $^2\!F_4(2)'$ and $ {^2\!G_2}(3^3)$. For these groups, we can verify $k(G)<P(S)/2$ using the ATLAS \cite{ATLAS}.

If $S$ is a Suzuki group, then \cite{suzuki} tells us that $k(S)=q+3$, and Lemma~\ref{l: sub conj} implies that $k(G)$ is at most $(q+3)f$, where $q=2^f$. On the other hand, $P(S)=q^2+1$ by \cite{GMPS}. Then $(q+3)f \geq \frac12 (q^2+1)$ if and only if $q=8$ (recall that $f$ is odd and $f\geq 3$). But if $q=8$, \cite{ATLAS} tells us that $S.3$ has 17 conjugacy classes and this case, too, is excluded.
\end{proof}

Next we deal with the case in which $G$ has socle $\PSL_d(q)$.

\begin{lemma}\label{l: psl}
Let $G$ be almost simple with socle $S\cong \PSL_d(q)$. Let $M$ be a core-free maximal subgroup of $G$, and write $n=|G:M|$. If $k(G)\geq \frac{n}{2}$, then one of the following holds.
\begin{enumerate}
    \item $G$ and $n$ are listed in Table~\ref{tab: psl ex}.
    \item $G\leq \PGammaL_d(q)$ and $M$ stabilizes a $1$-dimensional or a $(d-1)$-dimensional subspace of $\F_q^d$, thus $n=\frac{q^d-1}{q-1}$. 
\end{enumerate}
\end{lemma}

\begin{table}
    \centering
    \begin{tabular}{ccc}
   \hline\noalign{\smallskip}
$G$ & $n$ & $k(G)$  \\
         \noalign{\smallskip}\hline\noalign{\smallskip}
$\SL_2(4)=A_5$ & 6,10 & 5 \\
$\SL_2(4).2=S_5$ & 6,10 & 7 \\
$\PSL_2(5)=A_5$ & 5,10 & 5 \\
$\PGL_2(5)=S_5$ & 5,10 & 7 \\
$\PSL_2(7)$ & 7,7 & 6 \\
$\PSL_2(9)=A_6$ & 6,6 & 7 \\
$\PSL_2(9).2=S_6$ & 6,6,15,15 & 11 \\
$\PSL_2(11)$ & 11,11 & 8 \\
$\SL_3(2)$ & 8 & 6\\
$\SL_3(2).2$ & 8 & 9\\
$\SL_4(2)=A_8$  & 8, 28 & 14\\
$\SL_4(2).2=S_8$  & 8,28,35 & 22\\
          \noalign{\smallskip}\hline\noalign{\medskip}
    \end{tabular}
    \caption{Faithful primitive permutation representations of degree $n$ for almost simple groups $G$ with socle $\PSL_d(q)$ such that $k(G)\geq \frac n2$, and the action is not isomorphic to an action in (B) (see the remark after the statement of Lemma \ref{l: psl}).}
    \label{tab: psl ex}
\end{table}

Note that, in Table \ref{tab: psl ex}, in order to ensure that the lemma is true, $n=6$ appears for $G=\SL_2(4)$, but not for $G=\PSL_2(5)$ (even though $\SL_2(4)\cong \PSL_2(5)$). Similarly, $n=5$ appears for $\PSL_2(5)$, but not for $\SL_2(4)$. Similar considerations apply for the isomorphic groups $\PSL_2(7)$ and $\SL_3(2)$.

\begin{proof}
In this proof we use \cite{kantor}. The main theorem of this paper, together with Theorem \ref{c: fg}, implies that, if $k(G)\geq n/2$, then either $d\leq 4$, or
\begin{equation}\label{e: exc}
(d,q)\in\{(5,2), (5,4), (5,8), (6,2), (7,2)\},
\end{equation}
or $H:=M\cap \PGammaL_d(q)$ is reducible, or $H$ normalizes $\PSp_d(q)$. 

Assume first that $d\geq 5$. This rules out the case in which $H$ normalizes $\PSp_d(q)$. Let us now consider the case where $H$ is reducible, stabilizing a subspace of dimension $m$.

Assume first that $2\leq m \leq d-2$. Then $n=|G:M|>q^{m(d-m)}\geq q^{2d-4}$. If $\frac{n}{2}\leq 100q^\ell=100q^{d-1}$, then we get $q^{d-3}<200$. We want to whittle down the possibilities, as follows. \cite[Proposition 3.6]{FG1} states that $k(\PSL_d(q))\leq 2.5q^{d-1}$. This, together with the knowledge of $|\Out(S)|$ and Lemma \ref{l: sub conj}, reduces easily to the cases $(d,q)=(5,2), (5,3), (5,4), (6,2)$. The same argument and \cite{kantor} rules out the cases $(d,q)=(5,8),(7,2)$ in \eqref{e: exc}. We can deal with the remaining cases with GAP \cite{GAP}.

Assume now that $m\in\{1, d-1\}$. The case in which $G\leq \PGammaL_d(q)$ appears in item (2) of the statement. If $G\nleq \PGammaL_d(q)$, then $M$ is a novelty and $|G:M|\geq q^{2d-3}$, and the GAP calculation from the previous paragraph rules out all possibilities.

Let us turn, then, to study what happens when $d\in\{2,3,4\}$. We make use of the counts given in \cite{macdonald}.

When $d=2$, \cite{macdonald} implies that
\[
k(\PSL_2(q))=\frac{1}{(q-1,2)}\big(q+4(q-1,2)-3\big) \textrm{ and } k(\PGL_2(q)) = q+(2,q-1).
\]
We use this in combination with the explicit list of maximal subgroups in $\PSL_2(q)$ to conclude that either
\begin{enumerate}
    \item $q\leq 11$; or
    \item $q=16$ and $M$ is the normalizer of a torus, or a subfield subgroup such that $M\cap S\cong \PGL_2(\sqrt{q})$; or
    \item $q\in\{ 25, 49, 81, 64, 256\}$ and $M$ is a subfield subgroup such that $M\cap S\cong \PGL_2(\sqrt{q})$. 
\end{enumerate}
Using \cite{GAP} we get the possibilities in Table~\ref{tab: psl ex}.

Next assume that $d=3$. If $q$ is odd, then using \cite{macdonald} we see that $k(\PSL_3(q))\leq q^2+q$, and this, along with \cite{kantor}, allows us to conclude that $q\leq 9$. These possibilities can all be excluded using \cite{ATLAS}. If $q$ is even, then \cite{macdonald} implies that $k(G)\le 2f(q^2+q+10)$ where $q=2^f$. Using the list of subgroups in \cite{BHR} this is enough to conclude that $q\leq 16$. Now \cite{GAP} excludes the remainder.

Finally, assume that $d=4$. If $q$ is odd, then \cite{macdonald} implies that $k(G)\le 2f(q^3+q^2+5q+21)$ where $q=p^f$. We use \cite{kantor} to conclude that $q=3$. This final case is ruled out with \cite{ATLAS}. If $q$ is even, then \cite{macdonald} implies that $k(\SL_4(q))=q^3+q^2+q$ and, again, we use the list of subgroups in \cite{BHR} to conclude that $q\leq 16$.  Now \cite{GAP, BHR, ATLAS} rule out all except the listed exceptions for $q=2$.
\end{proof}

Finally we deal with almost simple classical groups with socle $S$  not isomorphic to $\PSL_d(q)$. Note that to deal with this class of groups it is sufficient to consider $S=\PSU_d(q)$ with $d\geq 3$ and $(d,q)\neq (3,2)$; $S=\PSp_d(q)$, with $d\geq 4$ and $(d,q)\neq (4,2)$; and $S=\POmega_d^\varepsilon(q)$ with $d\geq 7$.

\begin{lemma}\label{l: classical}
Let $G$ be almost simple with classical socle $S$, $S\not\cong \PSL_d(q)$. 
Let $M$ be a core-free maximal subgroup of $G$, and write $n=|G:M|$. If $k(G)\geq \frac{n}{2}$, then $G$ and $n$ are listed in Table~\ref{tab: as ex}.
\end{lemma}

\begin{table}
    \centering
    \begin{tabular}{ccc}
    \hline\noalign{\smallskip}
          $G$ & $n$ & $k(G)$  \\
          \noalign{\smallskip}\hline\noalign{\smallskip}
$\SOr_8^-(2)$ & 119 & 60 \\
$\Sp_8(2)$ & 120,136 & 81 \\
$\SOr_8^+(2)$ & 120 & 67 \\
$\Sp_6(2)$ & 28, 36 & 30 \\
$\PSp_4(3)=\PSU_4(2)$ & 27,36,40,40 & 20 \\
$\PSp_4(3).2=\PSU_4(2).2$ & 27,36,40,40,45 & 25 \\
$\PSU_4(3).(2\times 2)$ & 112 & 59 \\
$\PGammaU_4(3)$ & 112 & 61 \\
$\SU_3(3)$ & 28 & 14 \\
$\SU_3(3).2$ & 28 & 16 \\
          \noalign{\smallskip}\hline\noalign{\medskip}
         
    \end{tabular}
    \caption{Faithful primitive permutation representations of degree $n$ for almost simple classical groups $G$ with socle $S\not\cong\PSL_d(q)$ 
    such that $k(G)\geq \frac n2$.}
    \label{tab: as ex}
\end{table}


\begin{proof}
In order to exclude some potential examples, our basic strategy will be to use the bound $k(G)\leq |G:S|k(S)$ from Lemma \ref{l: sub conj}, and try to show that this is smaller than $P(S)/2$. In order to bound $k(S)$, we will use the results in \cite{FG1} for specific families, as follows.

\textbf{Suppose that $S\cong \PSU_d(q)$.} In this case \cite[Proposition 3.10]{FG1} implies that $k(S)\leq 8.26q^{d-1}$, and we use the values for $P(S)$ given in \cite{GMPS} to obtain that either $S$ is in
    \[
    \{\PSU_5(2), \PSU_6(2), \PSU_7(2), \PSU_5(3), \PSU_5(4)\}
    \]
    or else $d\leq4$. Groups with the five possible socles with $d>4$ can be ruled out using \cite{GAP}.
    
If $S=\PSU_4(q)$, then \cite{macdonald} implies that
\[
k(S)=\begin{cases}
\frac14(q^3+q^2+7q+23), & q\equiv 3\pmod 4; \\
\frac12(q^3+q^2+7q+9), & q\equiv 1\pmod 4; \\
q^3+q^2+3q+2, & q\equiv 0 \pmod 2.
\end{cases}
\]
Thus, in any case, $k(G)\leq 2f(q^3+q^2+7q+23)$ where $q=p^f$. Since $P(S)=(q+1)(q^3+1)$, by \cite{GMPS}, we conclude that $q\in\{2,3,4,5,8,16\}$.  If $q\leq 5$, then \cite{GAP} yields the listed cases. If $q\in\{8,16\}$, then $G=\PGammaU_4(q)$. The case $q=8$ is eliminated by \cite{magma}; we now consider the case $q=16$. 
We want to apply Lemma \ref{l: improvementclasses} with $G=\PGammaU_4(16)$ and $N=\SU_4(16)$. Consider the split torus $T$ of $N$ of order $17^3$, which intersects $284$ nontrivial $N$-classes.
By looking at eigenvalues, we see that none of these classes is fixed by the standard field automorphism $\sigma$ of order $8$ normalizing $T$. We deduce from Lemma \ref{l: improvementclasses} that $k(G)\leq 8(q^3+q^2+3q+2 -284/2)=34080$, which is enough to conclude that 
$k(G)<P(S)/2$.

If $G=\PSU_3(q)$, then \cite{macdonald} implies that
\[
k(S) = \begin{cases}
q^2+q+2, & q\not\equiv 2 \pmod 3; \\
\frac13 (q^2+q+12), & q\equiv 2\pmod 3.
\end{cases}
\]
Thus, in any case, $k(G)\leq 2f(q^2+q+12)$ where $q=p^f$. Since $P(S)=q^3+1$ if $q\neq 5$, by \cite{GMPS}, we conclude that $q\leq 9$ or $G=\PGammaU_3(16)$. For $q\leq 9$, we obtain the listed examples using \cite{GAP} and \cite{ATLAS}. For $G=\PGammaU_3(16)$, the same argument used for the case $\PGammaU_4(16)$ works.


\textbf{Suppose that $S\cong \PSp_d(q)$.} If $d\geq 6$, then we use \cite[Theorems 3.12 and 3.13]{FG1} along with the values for $P(S)$ given in \cite{GMPS} to conclude that $S$ is one of the following:
    \[
    \{\PSp_6(3), \Sp_6(2), \Sp_8(2), \Sp_{10}(2)\}.
    \]
We use \cite{GAP} and \cite{ATLAS} to check these cases and obtain the listed examples.

If $S=\PSp_4(q)$, then we use \cite{Wall} (for $q$ odd) and \cite{enomoto} (for $q$ even) to establish that
\[
k(\Sp_4(q))=\begin{cases} 
q^2+5q+10, & q \textrm{ odd;} \\
q^2+2q+3, & q \textrm{ even.}
\end{cases}.
\]
This, combined with \cite{GMPS}, implies that $q\leq 9$. Now \cite{GAP} and \cite{ATLAS} yield the listed examples.

\textbf{Suppose that $S\cong \POmega_{2\ell+1}(q)$.} Here we assume that $\ell \geq 3$ and that $q$ is odd. Now \cite[Theorem 3.19]{FG1} along with the values for $P(S)$ given in \cite{GMPS} imply that $S=\POmega_7(3)$. This final case can be excluded using \cite{ATLAS}.

\textbf{Suppose that $S\cong \POmega^\pm _{2\ell}(q)$ with $q$ odd.} We make use of \cite[Theorems 3.16 and 3.18]{FG1} along with the values for $P(S)$ given in \cite{GMPS} to obtain that
\[
S\in \{\POmega_{10}^\pm(3), \POmega_8^\pm(3), \POmega_8^+(5), \POmega_8^+(7)\}.
\]
In $\POmega_8^+(5)$ and $\POmega_8^+(7)$, the outer automorphism group is $S_4$, and a subgroup of $S_4$ has at most $5$ conjugacy classes, therefore by Lemma \ref{l: sub conj} we get $k(G)\leq 5k(S)$, which is enough to rule out these possibilities.

We use \cite{magma} and \cite{GAP} to rule out the cases where $S=\POmega_{10}^\pm(3)$ or $\POmega_8^\pm(3)$.

\textbf{Suppose that $S\cong \Omega^\pm _{2\ell}(q)$ with $q$ even.} We make use of \cite[Theorem 3.22]{FG1} along with the values for $P(S)$ given in \cite{GMPS} to obtain that
\[
S\in \{\Omega_{10}^\pm(2), \Omega_8^\pm(2), \Omega_8^+(4)\}.
\]
We use \cite{ATLAS} for the groups with $q=2$, and we get the listed examples. We can rule out $\Omega_8^+(4)$ using \cite{magma}.
\end{proof}

\subsection{Proof of Theorem \ref{t: main_almost_simple}}
\label{subsec: proof_thm_almost_simple}
Let $G$ be an almost simple primitive permutation group of degree $n$. Putting together Lemmas \ref{l: sporadic}, \ref{l: alternating}, \ref{l: exceptional}, \ref{l: psl} and \ref{l: classical}, we get that either $k(G)<n/2$, or we are in case (2) of Theorem \ref{t: main_almost_simple} (regarding Table \ref{tab: final_exceptions}, recall Remark \ref{r: clarifications}(iii)).

Note that, if the action of $G$ is isomorphic to an action in (B), then $k(G)<100n$ follows immediately from Theorem \ref{c: fg}. 

It remains to prove the asymptotic statement, that is, either $k(G)=O(n^{3/4})$, or the action of $G$ is isomorphic to an action in (A) or (B). We assume that this latter condition does not hold, and we want to show $k(G)=O(n^{3/4})$.

We may assume that $G$ is sufficiently large along the proof. Let $M$ be the stabilizer of a point in the action of $G$ on $n$ points; in particular $|G:M|=n$. Write $S=\Soc(G)$.

\textbf{Assume first} that $S\cong A_d$, and assume that $M$ is transitive on $d$ points; we will show that $k(G)=n^{o(1)}$ as $d$ tends to infinity. By Lemma \ref{l: partition} (or by the Hardy--Ramanujan asymptotic formula), we have $k(G)=O(1)^{\sqrt d}$. On the other hand, by Lemmas \ref{l: factorial} and \ref{l: primitive}, if $M$ is primitive on $d$ points then $n\geq (d/O(1))^d$; and by \eqref{eq:symm} in the proof of Lemma \ref{l: alternating}, if $M$ is imprimitive then $n\geq c^d$ for some constant $c$. Therefore $k(G)=n^{o(1)}$ if $S\cong A_d$.

\textbf{Assume now} that $S\cong \PSL_d(q)$. We have $k(G)=O(q^{d-1})$ by Theorem \ref{c: fg}. If $H:=M\cap \PGammaL_d(q)$ is reducible in the action on $\F_q^d$, one possibility is that it stabilizes a $k$-space for some $2\leq k \leq d-2$, and so $n> q^{2d-4}$. If $d\rightarrow \infty$, we see that $k(G)=o(n^{3/4})$; and if $d$ is bounded, we see that $k(G)=O(n^{3/4})$ (we actually have $k(G)=o(n^{3/4})$ as $q\rightarrow \infty$ except for the case $(d,k)=(4,2)$). The remaining possibility is that $G\nleq \PGammaL_d(q)$ and $M$ is the stabilizer of a flag (pair of incident point-hyperplane) or antiflag (pair of complementary point-hyperplane). But in this case $n>q^{2d-3}$, and the previous computation is sufficient for $d\geq 4$; and for $d=3$, $k(G)=O(n^{2/3})$.

If $d=2m\geq 4$ and $H$ normalizes $\PSp_{2m}(q)$, then
\[
n\geq \frac{1}{(m,q-1)}\cdot q^{m^2-m} (q^3-1)(q^5-1)\cdots (q^{2m-1}-1)
\]
and we can easily check that $k(G)=o(n^{3/4})$.

If now $H$ is irreducible and does not normalize $\PSp_d(q)$, we can apply the main theorem of \cite{kantor}. We see easily that $k(G)=o(n^{3/4})$ as $d\rightarrow \infty$. If $d$ is bounded instead, then we can assume that $q$ is large and in particular \cite{kantor} implies that $n\geq q^{(d-1)(d-2)/2}$, which proves $k(G)=o(n^{3/4})$ in case $d\geq 5$. In case $d\leq 4$, we can use the list of maximal subgroups of $\PSL_d(q)$ given in \cite{BHR} in order to prove $k(G)=o(n^{3/4})$ (if $H$ is irreducible, $n\gg q^{3/2}$ for $d=2$; $n\gg q^4$ for $d=3$; and $n\gg q^5$ for $d=4$).

\textbf{Assume finally} that $S$ is a group of Lie type and that $S\not\cong\PSL_d(q)$. In this case we want to show $k(G) = O(P(S)^{3/4})$, which implies the statement, since $P(S)\leq n$. This can be checked combining $k(G)=O(q^r)$ (where $r$ is the untwisted rank of $S$) with the value of $P(S)$ given in \cite{GMPS}. In fact, we get $k(G) = o(P(S)^{3/4})$ unless $S\cong \PSU_4(q)$. (We remark that, in the latter case, $P(S)$ is equal to the number of totally singular $2$-subspaces of $\F_{q^2}^4$; we also use \cite{BHR} in order to see that $n\gg q^5$ for every other primitive action of $G$.)

This concludes the proof of Theorem \ref{t: main_almost_simple}.

\begin{remark}
\label{r: littleo(n34)}
In Theorem \ref{t: main_almost_simple}(1), we actually showed that $k(G)=o(n^{3/4})$ as $n\rightarrow \infty$ unless $S\cong \PSL_4(q)$ and $G$ acts on the set of $2$-subspaces of $\F_q^4$, or $S\cong \PSU_4(q)$ and $G$ acts on the set of totally singular $2$-subspaces of $\F_{q^2}^4$. In these cases, we have $n\sim q^4$ and $k(S)\asymp q^3$, therefore $k(S)\asymp n^{3/4}$. (Recall that $f\asymp g$ means that $c_1f\leq g\leq c_2f$ for positive constants $c_1$ and $c_2$.) 
\end{remark}

\section{The general case}
\label{sec:general case}

In this section we prove Theorem \ref{t: conjugacyclassesprimitive}. We first prove a lemma.

\begin{lemma}
\label{l: 4k(A)square}
Let $G$ be a finite almost simple group with socle $S$. Then either $S\cong A_5, A_6, \PSL_2(7), \PSL_2(11)$, or $4\cdot k(G)^2 < |S|$.  Moreover, $k(G)^3 =O(|S|)$.
\end{lemma}
We note that we actually have $k(G)^3=o(|S|)$ as $|S|\rightarrow \infty$, except for the case $S\cong \PSL_2(q)$.

\begin{proof}
We first prove $4\cdot k(G)^2 < |S|$, with the listed exceptions.

\textbf{Assume first} that $S\cong A_d$. Then Lemmas \ref{l: partition}, \ref{l: factorial} and a straightforward computation imply that it is sufficient to show
\[
3.2\cdot e^{5.2\sqrt d +d} < d^{d+2},
\]
which is true for $d\geq 10$. For $d\leq 9$, direct check gives the exceptions in the statement.

\textbf{Assume now} that $S\cong \PSL_d(q)$. Using the bound $k(S)\leq 2.5 q^{d-1}$ from \cite{FG1}, Lemma \ref{l: sub conj}, and the fact that $|G:S|\leq 2f(d,q-1)$, with $q=p^f$, we see that it is sufficient to show
\[
100f^2(d,q-1)^3 < q^{d(d-1)/2-2d+2} (q^2-1)\cdots (q^d-1).
\]
If $d\geq 4$, we can easily verify that this is true. For $d=3$, \cite{macdonald} tells us that $k(S)\leq q^2+q$. We compute that it is enough to show
\[
16f^2(3,q-1)^3(q+1) < q(q-1)(q^3-1),
\]
which can be verified unless $q=2,4$. The case $q=2$ is in the statement (since $\SL_3(2) \cong \PSL_2(7)$), while the case $q=4$ can be excluded with \cite{GAP}.

If $d=2$, we use the exact value of $k(S)$ (recalled in the proof of Lemma \ref{l: psl}), in order to reduce to the cases $q\leq 16$ or $q=25,27,32,64,81,128,256$. Then we use \cite{GAP} and we get the cases $q=4,5,7,9,11$ in the statement.

\textbf{Assume} that $S$ is classical and that $S\not\cong \PSL_d(q)$. Here one can prove that $4k(G)^2 < |S|$ using the upper bounds for $k(S)$ given in \cite{FG1}. One can also argue as follows (but this is not necessary). 
If $G$ appears in Table \ref{tab: final_exceptions}, we can make a direct check. If $G$ is not in Table \ref{tab: final_exceptions}, then Theorem \ref{t: main_almost_simple} tells us that $k(G)<P_m(G)/2$, where $P_m(G)$ denotes the smallest index of a core-free maximal subgroup of $G$. Now it is known (see \cite[p. 178]{KL}) that $P(S)\leq |S|^{1/2}$. In particular, whenever $P_m(G)=P(S)$, we can immediately conclude $4k(G)^2 < |S|$. Certainly we have $P(S)\leq P_m(G)$. Using the value of $P(S)$ given in \cite{GMPS} (see also \cite{coopersteinminimal}, where an explicit $M$ for which $|S:M|=P(S)$ is given), and consulting \cite{KL}, we deduce that $P_m(G)=P(S)$ unless $S\cong \PSU_3(5)$, $S\cong \Sp_4(q)$ with $q$ even, $S\cong \POmega_8^+(q)$, or $S\cong \POmega^+_{2m}(3)$ with $m\geq 4$. (If $S\cong \PSU_3(5)$, $|S:M|=P(S)$ where $M$ is isomorphic to $A_7$; if $S\cong \POmega^+_{2m}(3)$, $M$ is the stabilizer of a nondegenerate $1$-space.) We can exclude the unitary case with \cite{ATLAS}; in the symplectic case we can use $k(\Sp_4(q))=q^2+2q+3$ (see the proof of Lemma \ref{l: classical}); in the orthogonal cases we can use the bound $k(\POmega_{2m}^+(q))\leq 14q^m$ given in \cite{FG1}. 

\textbf{Assume} that $S$ is exceptional. In the proof of Lemma \ref{l: exceptional} we actually proved $k(G)<P(S)/2$, 
therefore we conclude by the argument of the previous paragraph.

\textbf{Assume finally} that $S$ is sporadic. We use \cite{ATLAS} to conclude $4k(G)^2<|S|$.

It remains to prove the asymptotic statement, that is, $k(G)^3 =O(|S|)$ (and indeed $k(G)^3 =o(|S|)$ if $S\not\cong \PSL_2(q)$). We may assume that $S$ is sufficiently large, and the statement is easy to check, using Lemma \ref{l: partition} and Theorem \ref{c: fg}.
\end{proof}

We need three technical lemmas.

\begin{lemma}
\label{l: kB leq kA}
Assume that $S\cong A_d$, or that $S$ is the socle of some group appearing in Table \ref{tab: final_exceptions}. If $S\leq B\leq A\leq \Aut(S)$, then $k(B)\leq k(A)$, unless $A=\Omega^+_8(2).S_3$.
\end{lemma}

\begin{proof}
If $S\cong A_d$, the statement follows from Lemma \ref{l: alt sym}, and by direct check in case $d=6$. If $S$ is the socle of some group appearing in Table \ref{tab: final_exceptions}, we use \cite{ATLAS}.
\end{proof}

\begin{lemma}
\label{l: flags}
Assume that $G$ is almost simple with socle $S\cong\PSL_d(q)$, with $d\geq 3$, and let $m$ denote the number of flags (that is, pairs of incident point-hyperplane) in $\F_q^d$. Then, $k(G)<m/2$ and $k(G)=O(m^{2/3})$.
\end{lemma}

We note that we actually have $k(G)=o(m^{2/3})$ as $m\rightarrow \infty$, except in case $d=3$.

\begin{proof}
We begin with the inequality $k(G)<m/2$. If $G\nleq \PGammaL_d(q)$, then $G$ acts primitively on the set of flags, and the statement follows from Theorem \ref{t: main_almost_simple}. Assume now that $G\leq \PGammaL_d(q)$. Then $G$ acts primitively on the set of $2$-subspaces of $\F_q^d$. It is easy to see that the number of $2$-subspaces is smaller than the number of flags. Assume that $d\geq 4$. Then, by Lemma \ref{l: psl}, either $k(G)<m/2$, or $G$ appears in Table \ref{tab: psl ex}. Examining Table \ref{tab: psl ex}, we see that $k(G)<m/2$ also in the latter case.

We are left with the case $d=3$. We have $k(G)\leq 100q^2$ by Theorem \ref{c: fg}, and moreover $m>q^3$. In particular, if $k(G)\geq m/2$ then $q<200$. We whittle down a bit the possibilities. Write $a=(3,q-1)$. By \cite{macdonald} and Lemma \ref{l: sub conj}, we deduce $k(G)\leq |G:S|\cdot (q^2+q+5a-5)/a < 2|G:S|\cdot q^2$. Therefore, if $q=p^f$, we have $q<8af$. Using $q<200$, we see that we are reduced to the cases $q\leq 27$ and $q=32,64$, which can be checked with \cite{GAP} (if $q\neq 2,4,8,16$, it is enough to show that $4f(q^2+q+5a-5)$ is smaller than $m$, without computing the actual value of $k(G)$).

The asymptotic statement $k(G)=O(m^{2/3})$ can be checked easily using $k(G)=O(q^{d-1})$.
\end{proof}

\begin{lemma}
\label{l: kB < m/2}
Let $A$ be an almost simple primitive group of degree $m$ with socle $S$, and assume that $A$ is not in the possibilities of Theorem \ref{t: main_almost_simple}(2). Let $S\leq B\leq A$. Then, $k(B)<m/2$. Moreover, for every fixed $\alpha > 3/4$, if $S$ is sufficiently large then $k(B)<m^\alpha$.
\end{lemma}

\begin{proof}
We begin with the inequality $k(B)<m/2$. Write $m=|A:M|$ for some core-free maximal subgroup $M$ of $A$. If $B=A$, the claim is true by Theorem \ref{t: main_almost_simple}. Assume, for a contradiction, that there exists $B$ such that $k(B)\geq m/2=|B:B\cap M|/2$. Let $T$ be a core-free subgroup of $B$, maximal with respect to the property that $B\cap M\leq T$ and that $T$ is core-free in $B$ (that is, $T$ does not contain $S$).

Note that the subgroups of $B$ properly containing $T$ must contain $S$. Then choose $C$ such that $T<C\leq B$ and $T$ is maximal in $C$. In particular, $C$ acts primitively on the cosets of $T$, and moreover, by Lemma \ref{l: sub conj},
\[
|B:C|\cdot k(C)\geq k(B)\geq \frac{|B:C||C:T|}{2},
\]
whence $k(C)\geq |C:T|/2$. Therefore we can apply Theorem \ref{t: main_almost_simple}. The first possibility is that $C$ appears in Table \ref{tab: final_exceptions}. By Lemma \ref{l: kB leq kA} and $k(B)\geq m/2$, we deduce $A=\Omega^+_8(2).S_3$. Then by \cite{ATLAS} $m\geq 3600$, which contradicts $k(B)\geq m/2$. By Lemma \ref{l: kB leq kA}, we also see that it cannot be $S\cong A_d$. By Theorem \ref{t: main_almost_simple} and Lemma \ref{l: psl}, the only remaining possibility is that $S\cong \PSL_d(q)$, $C\leq \PGammaL_d(q)$ and $T$ is the stabilizer of a $1$-space or $(d-1)$-space. In particular, $B\cap M$ stabilizes a $1$-space or a $(d-1)$-space.

Assume first that $A\leq \PGammaL_d(q)$. By assumption, $M$ is not the stabilizer of a $1$-space or $(d-1)$-space. Then, there is no other possibility for $M$ (in such a way that $B\cap M$ fixes a $1$-space or $(d-1)$-space), which is a contradiction. Assume finally that $A\nleq \PGammaL_d(q)$. Then the only possibility is that $M$ is the stabilizer of a flag or antiflag. In particular, $m$ is larger than the number of flags in $\F_q^d$, which contradicts Lemma \ref{l: flags}. This final contradiction proves that $k(B)<m/2$ for every $S\leq B\leq A$.

Now we want to show that, for every fixed $\alpha>3/4$, if $S$ is sufficiently large then $k(B)<m^\alpha$ for every $S\leq B\leq A$.

By Theorem \ref{t: main_almost_simple}, we have $k(A)=O(m^{3/4})$. Assume that $k(B)\geq m^\alpha$. We want to show that $S$ has bounded order (in other words, we want to show that, if $S$ is sufficiently large, we get a contradiction). By taking $S$ large, we have $k(B)>k(A)$. Much of the argument of the first part of the proof carries unchanged, except that we have the inequality
\[
|B:C|\cdot k(C)\geq k(B)\geq |B:C|^\alpha \cdot |C:T|^\alpha,
\]
from which $k(C)\geq |C:T|^\alpha \cdot |B:C|^{\alpha-1}$. Note that $|B:C|\leq |\Out(S)|$ and $|C:T|\geq P(S)$. Using \cite[Table 4]{GMPS}, we easily see that $|\Out(S)|=P(S)^{o(1)}$ as $|S|\rightarrow \infty$ (the statement being obvious in case $S\cong A_d$), from which we get that, for every fixed $\beta<\alpha$,
\[
k(C)\geq |C:T|^\alpha \cdot |B:C|^{\alpha-1}>|C:T|^\beta
\]
if $S$ is sufficiently large. In particular we may take $\beta>3/4$, and by Theorem \ref{t: main_almost_simple}, we deduce that $C$ and $|C:T|$ must appear in item (2) of the theorem. Then, the argument that we used in the first part of the proof, together with Lemma \ref{l: flags}, gives a contradiction. 
\end{proof}

\subsection{Proof of Theorem \ref{t: conjugacyclassesprimitive}} We can now prove Theorem \ref{t: conjugacyclassesprimitive}. We will apply many times Lemma \ref{l: sub conj}, usually with no mention. Moreover, we will often use the following theorem from \cite{liebeckpyber}, which we recalled in the introduction.

\begin{theorem}
\label{t: pyber_kP leq 2^r-1}
Let $r\geq 1$ and let $P\leq S_r$. Then, $k(P)\leq 2^{r-1}$.
\end{theorem}

Let $G$ be a primitive permutation group of degree $n$ with nonabelian socle $\Soc(G)\cong S^r$, with $S$ simple.

In the following proof, a permutation group $G$ of degree $n$ in \textit{product action} refers to a group $G\leq A\wr S_r$, where $A$ is almost simple primitive on $m$ points with socle $S$ and $G$ acts on $n=m^r$ points (so we do not include the actions that sometimes are called \textit{holomorph compound} and \textit{compound diagonal}; see \cite{LPS, Praeger} for descriptions and terminology for finite primitive permutation groups).

\begin{proof}[Proof of Theorem \ref{t: conjugacyclassesprimitive}]

\textbf{Assume first} that the action of $G$ is not product action; we want to show $k(G)< n/2$ and $k(G)=O(n^\delta)$ for some absolute $\delta< 1$. We begin with the first inequality.

We have $r\geq 2$ and either $n=|S|^r$, or $r=\ell t$ with $\ell \geq 2$, $t \geq 1$ and $n=|S|^{(\ell-1)t}$. In particular $n \geq |S|^{r/2}$. Furthermore, $G\leq \Aut(S)\wr S_r$. Then, by Lemma \ref{l: sub conj} and Theorem \ref{t: pyber_kP leq 2^r-1}, $k(G)\leq k(G\cap \Aut(S)^r)\cdot 2^{r-1}$. Now $G\cap \Aut(S)^r$ admits a normal series of length $r$ in which every factor is almost simple with socle $S$; therefore, by Theorem \ref{t: pyber_kP leq 2^r-1}, $k(G\cap \Aut(S)^r)\leq f(S)^r$, where $f(S)=\text{max}\{k(A): S\leq A\leq \Aut(S)\}$. We deduce that it is enough to show that
\[
2f(S) < |S|^{1/2}.
\]
By Lemma \ref{l: 4k(A)square}, this is true unless $S\cong A_5, A_6, \PSL_2(7), \PSL_2(11)$. Assume then that we are in one of these cases. If $n=|S|^r$ or $n=|S|^{(\ell-1)t}$ with $\ell \geq 3$, then $n\geq |S|^{2r/3}$, hence by the same argument as above we have $k(G)<n/2$ provided
\[
2f(S) < |S|^{2/3}.
\]
We can check that this is true. Therefore we are reduced to the case in which $S\in \{A_5, A_6, \PSL_2(7), \PSL_2(11)\}$, $r=2t$ and $n=|S|^t$.

Assume first that $t=1$, and let $h(S)$ be the maximum number of conjugacy classes of a primitive group on $|S|$ points with socle $S^2$. We can use \cite{GAP} in order to compute that $h(S) < |S|/2$. 

Next we deal with any $t\geq 1$. We have $G\leq D \wr S_t$, where $D$ has socle $S^2$ and is primitive on $|S|$ points. Then $k(G)\leq k(G\cap D^t)\cdot 2^{t-1}$. Now $G\cap D^t$ admits a normal series of length $t$ in which every factor has socle $S^2$ and is primitive on $|S|$ points; 
in particular $k(G\cap D^t)\leq h(S)^t<(|S|/2)^t$ and therefore $k(G)<|S|^t/2=n/2$, as wanted.

We turn now to the asymptotic statement; namely, $k(G)=O(n^\delta)$ for an absolute $\delta<1$. We assume that $n$ is sufficiently large and we show $k(G)\leq n^\delta$ (which is equivalent up to enlarging $\delta$). We will show in various places that $k(G)\leq n^{\delta'}$ for various $\delta'$. In order to simplify notation, we will always use the same symbol $\delta$ -- one should just take the maximum.  

Assume first that $S$ is sufficiently large. By Lemma \ref{l: 4k(A)square}, we have $f(S) < |S|^{0.35}/2$. Using $n \geq |S|^{r/2}$, we deduce $k(G)<n^{0.7}$.

Assume now that $S$ has bounded order. If $S\not\cong A_5, A_6, \PSL_2(7), \PSL_2(11)$, by Lemma \ref{l: 4k(A)square} we have $2f(S) < |S|^{1/2}$, and in particular
\[
k(G)< (2 \cdot f(S))^r < |S|^{r\delta/2}\leq n^\delta
\]
for some $\delta <1$ absolute (since $|S|$ is bounded).

Assume then that $S\cong A_5, A_6, \PSL_2(7), \PSL_2(11)$. If $n=|S|^r$ or $n=|S|^{(\ell-1)t}$ with $\ell \geq 3$, then $n\geq |S|^{2r/3}$ and, as already observed, $f(S)<|S|^{2/3}/2$; therefore the same argument as above applies. The remaining case is $l=2$ and $r=2t$. We already observed that $2h(S)<|S|$, from which we get 
\[
k(G)< (2 \cdot h(S))^t < |S|^{t\delta}= n^\delta
\]
for some $\delta<1$ absolute.

\textbf{Assume now} that the action of $G$ is product action, and assume that we are not in case (2) of the statement. We want to show $k(G)<n/2$ and $k(G)=O(n^\delta)$ for some $\delta<1$ absolute. 
We begin with the first inequality. We have $G\leq A\wr S_r$, $n=m^r$, and $A$ is an almost simple group with socle $S$ admitting a primitive action on $m$ points, which is not among the possibilities of Theorem \ref{t: main_almost_simple}(2).

Note that $k(G)\leq k(G\cap A^r)\cdot 2^{r-1}$, and $G\cap A^r$ admits a normal series of length $r$ in which each factor is isomorphic to a subgroup $S\leq B\leq A$. By Lemma \ref{l: kB < m/2}, $k(B)<m/2$ for every $S\leq B\leq A$, and therefore $k(G\cap A^r)<(m/2)^r$ and $k(G)< n/2$, as wanted.

The asymptotic statement $k(G)=O(n^\delta)$ for some $\delta<1$ is proved as we did for the case in which $n=|S|^r$ or $n=|S|^{(\ell-1)t}$, dividing the cases $|S|$ sufficiently large and $|S|$ bounded. If $S$ is sufficiently large, by Lemma \ref{l: kB < m/2} we have $k(B)< m^{0.8}/2$ for every $B\leq S\leq A$, and therefore $k(G)< n^{0.8}/2$. If $S$ has bounded order, we only need to use $k(B)<m/2$ for every $S\leq B\leq A$, which holds again in view of Lemma \ref{l: kB < m/2}.

\textbf{Assume now} that we are in case (2)(i) of the statement; we want to show $k(G)<n^{1.31}$. We have $G\leq A\wr S_r$ and $A$ is almost simple acting primitively on $m$ points.

Let us consider first the case in which $A=M_{12}$ acting primitively on $m=12$ points. If $r\geq 4$, \cite{garonzimaroti} tells us that a subgroup of $S_r$ has at most $5^{(r-1)/3}< 5^{r/3}$ conjugacy classes. In particular, using that $k(A)= 15$, we deduce that $k(G)< 15^r \cdot 5^{r/3}$, which we verify to be at most $n^{1.31}$. If $r\leq 3$, we use that a subgroup of $S_r$ has at most $r$ conjugacy classes, so $k(G)\leq 15^r \cdot r$, which is less than $n^{1.31}$ for $r\leq 3$.

Let us consider now all other cases. By Lemma \ref{l: kB leq kA} we have $k(B)\leq k(A)$ for every $S\leq B\leq A$. Then $k(G)\leq k(A)^r\cdot 2^{r-1}< (2k(A))^r$, so we only need to show that $2k(A)\leq m^{1.31}$. This can be checked easily going through all cases in Table \ref{tab: final_exceptions} (but leaving out the case of $M_{12}$ acting on $12$ points).

\textbf{Assume finally} that we are in case (2)(ii) of the statement, and the action of $A$ is isomorphic to an action in (B); in particular $m=(q^d-1)/(q-1)$. We want to show $k(G)<n^{1.9}$.

If $r\geq 4$, by Theorem \ref{c: fg} we have
\[
k(G)\leq (100q^{d-1})^r\cdot 5^{(r-1)/3}< (100\cdot 5^{1/3})^r\cdot n,
\]
hence we are done provided $100\cdot 5^{1/3}\leq m^{0.9}$, that is, $m\geq 303$. If $r\leq 3$, we use $k(G)\leq (100q^{d-1})^r\cdot r$, and we see that $m\geq 303$ is enough also in these cases. 

Therefore we assume that $m<303$; this leaves us with the cases $d=6,7,8$ and $q=2$; or $d=5$ and $q\leq 3$; or $d=4$ and $q\leq 5$, or $d=3$ and $q\leq 16$; or $d=2$ and $q<302$.

We whittle down slightly the possibilities for $d=2$. In the proof of Lemma \ref{l: psl}, we recalled the exact value of $k(\PSL_2(q))$ and $k(\PGL_2(q))$. Using this and $q<302$, it is easy to deduce that $k(A)\leq 8(q+1)=8m$. By the same computation as above, we are done provided $8\cdot 5^{1/3}\leq m^{0.9}$, that is, $m\geq 19$. Therefore if $d=2$ then we may assume that $q\leq 17$.

Now we deal with all the remaining cases (for $d\leq 8$). We only need to show that $k(A)\cdot 5^{1/3}\leq m^{1.9}$, which can be checked with \cite{GAP}.
\end{proof}


\section{Further comments}
\label{sec: final_comments}

\subsection{Theorem \ref{t: conjugacyclassesprimitive}(2)(i)}

In Theorem \ref{t: conjugacyclassesprimitive}(2)(i), we proved $k(G)<n^{1.31}$. Can we get better bounds? Since we have finitely many possibilities for the almost simple primitive group $A$ of degree $m$, we fix $A$ and $m$, and we want to estimate $k(G)$ where $G\leq A\wr S_r$ is primitive, mainly when $r$ is large.

First, we show that it is not always true that $k(G)=o(n)$ as $n\rightarrow \infty$ (and in fact it is not even true that $k(G)=O(n)$).

\begin{lemma}
\label{r: example_M12}
Consider $A=M_{12}$ acting primitively on $12$ points, and consider $G=A\wr C_r$ acting on $n=12^r$ points, where $C_r$ is cyclic of order $r$. If $r$ is large enough, then $k(G)>n^{1.08}$.
\end{lemma}

\begin{proof}
We have 
\[
k(G)\geq \frac{k(A)^r}{r}.
\]
Since $k(A)=15$, this is easily seen to be larger than $n^{1.08}$ for $r$ large enough.
\end{proof}

The same argument shows that $k(G)>n^\alpha$ for some absolute $\alpha>1$ whenever $A$ and $m$ in Table \ref{tab: final_exceptions} are such that $k(A)>m$ (but in the table, $A$ and $m$ are replaced by $G$ and $n$). This happens rarely; specifically, when
\[
(A,m)\in \{(M_{12}, 12), (M_{24},24), (\Sp_6(2),28)\}.
\]
Let us consider now the case in which $k(A)\leq m$ (by looking at Table \ref{tab: final_exceptions}, this is equivalent to $k(A)<m$). By Lemma \ref{l: kB leq kA}, we have $k(B)\leq k(A)$ for every subgroup $S=\Soc(A)\leq B\leq A$. We assume that $r\geq 4$, so that by \cite{garonzimaroti} a subgroup of $S_r$ has at most $5^{(r-1)/3}< 5^{r/3}$ conjugacy classes. Then, we have $k(G)< (k(A)5^{1/3})^r$, and whenever $k(A)\cdot 5^{1/3}<m$ we get  $k(G)< n^{\delta}$ for some absolute $\delta<1$. The condition $k(A)\cdot 5^{1/3}<m$ holds in some cases, but not quite in all.

Therefore one should try to change the argument. We make the following conjecture.

\begin{conjecture}
\label{conj_o(n)}
Let $A$ be an almost simple primitive group on $m$ points appearing in Table \ref{tab: final_exceptions}, and assume that $k(A)<m$. Then, for every primitive subgroup $G\leq A\wr S_r$ on $n=m^r$ points, $k(G)=o(m^r)$ as $r\rightarrow \infty$.
\end{conjecture}

In order to address Conjecture \ref{conj_o(n)}, it seems relevant to estimate the number of conjugacy classes in wreath products (although $G$ need not be a full wreath product, which is a complication).

\subsection{Conjugacy classes in wreath products} Let $A\neq 1$ be a finite group, and let $P$ be a transitive permutation group of degree $r$. Throughout, denote $k=k(A)$. Consider the wreath product $G=A\wr P$.
By Theorem \ref{t: pyber_kP leq 2^r-1}, we have $k(G)\leq k^r\cdot 2^{r-1}$. Does a considerably better bound hold? If necessary, we may imagine that $A$ is fixed and $r\rightarrow \infty$. In fact, we ask a question which is independent of the relation between $A$ and $r$.

\begin{question}
\label{q: wreath}
Let $A\neq 1$ be a finite group, let $P\leq S_r$ be transitive, and set $G=A\wr P$. Is $k(G)=O(k^r)$?
\end{question}

We should note that a positive answer to Question \ref{q: wreath} would not necessarily provide a positive answer to Conjecture \ref{conj_o(n)} (since, in Conjecture \ref{conj_o(n)}, $G$ needs not be a wreath product).


The next lemma gives an affirmative answer to Question~\ref{q: wreath} for the case where $P\leq S_r$ is regular. Before proving the lemma, we recall the combinatorial description of the conjugacy classes of $G=A\wr P$, in general: View the $k$ conjugacy classes of $A$ as $k$ distinct colours. Let $\pi_1, \ldots, \pi_t$ be representatives for the conjugacy classes of $P$. For each $i$, colour the cycles of $\pi_i$ in each possible way, and identify two colourings if one is obtained from the other by conjugation in $\C_P(\pi_i)$ (note that $\C_P(\pi_i)$ acts on the cycles of $\pi_i$). In this way we get the conjugacy classes of $G=A\wr P$; these can be thought of as the conjugacy classes of $P$, in which each cycle has a colour, and two colourings are identified as described above.

\begin{lemma}
\label{l: regular}
Let $G=A\wr P$ with $A\neq 1$, $P\leq S_r$ regular, and set $k=k(A)$. Then 
\[
k(G)=\frac{k^r}{r} +O(rk^{r/2}).
\]
\end{lemma}

\begin{proof}
Assume that $\pi\in P$ has order at least $2$; then $\pi$ has at most $r/2$ cycles. Summing over all nontrivial elements $\pi\in P$, we deduce that the number of colourings of the cycles of all nontrivial elements $\pi\in P$ is at most $rk^{r/2}$.

Now we consider the colourings of the cycles of the identity element $1\in P$. The action of $P=\C_P(1)$ on the cycles can clearly be identified with the action of $P$ on the set $\{1, \ldots, r\}$.

Let $\mathcal C$ be a colouring of $\{1, \ldots, r\}$. The size of the $P$-orbit of $\mathcal C$ is strictly smaller than $r$ if and only if $\mathcal C$ is stabilized by a nontrivial element $\pi\in P$, which implies that $\mathcal C$ has constant colours along the cycles of $\pi$. Therefore, the number of such colourings is at most $rk^{r/2}$. This implies that the number of colourings whose $P$-orbit has size $r$ is at least $k^r-rk^{r/2}$, whence
\[
k(G)=\frac{k^r}{r}+O(rk^{r/2}).
\]
This proves the lemma.
\end{proof}

\subsection{Theorem \ref{t: conjugacyclassesprimitive}(2)(ii)}

In this case we have $G\leq A\wr S_r$ where $A$ is almost simple primitive on $m$ points. Work of Mar\'oti \cite{marotibounding} tells us that $k(G)\leq p(n)$ and this bound is achieved if the action of $A$ is isomorphic to an action in (A). 

Assume instead that the action of $A$ is isomorphic to an action in (B). We have shown that, in this case, $k(G)<n^{1.9}$. This is certainly a long way from being sharp; let us consider what might be possible. 

First, recall that, if $q$ is odd, and if $A=\PGL_2(q)$, then $k(A)=q+2>m=q+1$. If we take for instance $q=5$ then, by the same argument as in Lemma \ref{r: example_M12}, we see that $k(A\wr C_r)>n^{1.08}$ for $r$ sufficiently large. Therefore it is not true in general that $k(G)=O(n)$.

However, in the other direction, observe that, for any $G$ in the case under consideration, the usual bound $k(G)<(100q^{d-1})^r \cdot 2^r$ implies that, for every fixed $\eps>0$,  $k(G)< n^{1+\eps}$ provided $\PSL_d(q)$ is sufficiently large (or, equivalently, provided $m$ is sufficiently large). We are left with the natural question:

\begin{question}
Let $A$ be an almost simple primitive group isomorphic to a group in (B), and assume that $G\leqslant A\wr S_r$ is primitive on $n=m^r$ points. What is the minimum value of $\eps$ such that $k(G)<n^{1+\eps}$?
\end{question}

\bibliography{references}
\bibliographystyle{alpha}

\end{document}